\documentclass[11pt]{article}
\usepackage{amsmath,amssymb,amsthm,enumerate}
\usepackage[all,dvips]{xy}

\newcommand{\grad}{\operatorname{grad}}

\newcommand{\stwoalpha}{S_{2,\alpha}(-2)}

\newcommand{\mmu}{^{\mu}}
\newcommand{\nnu}{^{\nu}}

\newcommand{\mathC}{\mathbb{C}}
\newcommand{\mathH}{\mathbb{H}}

\newtheorem{definition}{Definition}
\newtheorem{lemma}[definition]{Lemma}
\newtheorem{theorem}[definition]{Theorem}

\newtheorem*{thrm}{Theorem}
\newtheorem*{coro}{Corollary}

\begin{document}

\title{Equiboundedness of the Weil-Petersson metric}         
\author{Scott A. Wolpert\footnote{2010 Mathematics Subject Classification Primary: 32G15, 30F60.}}        
\date{\today}          
\maketitle

\begin{abstract}  Given a topological type for surfaces of negative Euler characteristic, uniform bounds are developed for derivatives of solutions of the $2$-dimensional constant negative curvature equation and the Weil-Petersson metric for the Teichm\"{u}ller and moduli spaces.  The dependence of the bounds on the geometry of the underlying Riemann surface is studied.  The comparisons between the $C^0$, $C^{2,\alpha}$ and $L^2$ norms for harmonic Beltrami differentials are analyzed.  Uniform bounds are given for the covariant derivatives of the Weil-Petersson curvature tensor in terms of the systoles of the underlying Riemann surfaces and the projections of the differentiation directions onto {\it pinching directions}.  The main analysis combines Schauder and potential theory estimates with the analytic implicit function theorem.  
\end{abstract}

\section{Introduction.}

For a compact Riemann surface possibly with punctures, uniformized by the hyperbolic plane, the Bers embedding represents Teichm\"{u}ller space as a bounded domain in complex number space.   The Bers embedding of Teichm\"{u}ller space is defined by lifting a quasiconformal homeomorphism of Riemann surfaces to have domain the upper half plane $\mathbb H$, extending by a conformal map of the lower half plane $\mathbb L$ to a homeomorphism of $\mathbb C$, and computing the Schwarzian derivative on $\mathbb L$ \cite{Bersembed}.  The Schwarzian derivative depends only on the boundary values of the quasiconformal homeomorphism, the equivalence class in Teichm\"{u}ller space.  The embedding maps into the space of holomorphic quadratic differentials for the complex conjugate of the base Riemann surface with the supremum norm relative to the hyperbolic metric.  For the given norm, the embedding has a special uniform property - the image contains the ball of radius $1/2$ and is contained in the ball of radius $3/2$ \cite{Lehtobk}.  

The Bers embedding has a close relationship to the Weil-Petersson metric.  A local coordinate at the origin for a K\"{a}hler metric is $\omega$-normal provided the metric has all vanishing pure-$z$ and pure-$\bar z$ derivatives at the origin.  Provided it exists, the $\omega$-normal coordinate with a given tangent frame at an initial point is unique.  The Bers embedding is $\omega$-normal for the Weil-Petersson metric \cite[Theorem 4.5]{Wlbers}.  We now extend the earlier analysis to show that the Weil-Petersson metric has a uniform property in the Bers embedding.  The uniform property is a generalization of the radius $\frac12$ and $\frac32$ results.  The range of the Bers embedding also has the description as the space of harmonic Beltrami differentials. We do not consider the dependence of estimates on topological type.  

\begin{thrm} For an initial $L^2$-orthonormal basis of harmonic Beltrami differentials, the derivatives of the Weil-Petersson metric tensor at the origin of the Bers embedding are uniformly bounded depending only on the differentiation order.
\end{thrm}

Weil-Petersson displacement is calibrated by the $L^2$-metric for harmonic Beltrami differentials, while the Bers embedding is calibrated by the supremum norm.  We find that the magnitude of derivatives of the Weil-Petersson metric is given solely by the comparison of $L^2$ and supremum norms.  In \cite[Corollary 11]{Wlbhv}, we found that the comparison of the two norms is given in terms of the projection onto the gradients of the small geodesic-length functions; see Definition \ref{Ratio} and Lemma \ref{norm2} below.  The maximal value for the comparison is $systole^{-1/2}$, for systole the smallest geodesic-length.  The minimal value for the comparison is an absolute constant.  In particular for compact subsets of the moduli space of Riemann surfaces and for tangent directions parallel to the compactification divisor of the moduli space, the $L^2$ and supremum norms are uniformly comparable. 

\begin{coro} A $k^{th}$ derivative of the Weil-Petersson metric tensor is bounded by $systole^{-k/2}$.  The covariant derivatives of Weil-Petersson curvature are uniformly bounded on compact subsets of the moduli space and for directions parallel to the compactification divisor of the moduli space.
\end{coro}
The expansions for the Weil-Petersson covariant derivative of geodesic-length \cite[Theorem 4.7]{Wlbhv} and curvature \cite[Theorem 15]{Wlcurv} show that the bounds are optimal for the first and second derivatives. 

The main quantities to analyze are the variation of the hyperbolic metric and the variation of the projection onto harmonic Beltrami differentials.  Defining equations for the quantities are given in terms of the Laplacian and the implicit function theorem is applied for each equation. Initial  inversion of the linearized defining equations is given by Green's operators.  For the variation of the hyperbolic metric, inversion in the appropriate Banach space is given by applying the $C^{2,\alpha}$ Schauder estimates. For the projection onto harmonic Beltrami differentials, Sobolev space estimates are applied.  

The present estimates should be generally qualitatively optimal.  Strengthening the uniform boundedness of derivatives of the metric in the Bers embedding would involve explicit constants or applying specific information about tangent directions.  The integrand for the Weil-Petersson metric is the product of projections of harmonic Beltrami differentials and the reciprocal of the hyperbolic metric.  We estimate the projections in the Sobolev space $H^1\subset L^2$ and the hyperbolic metric in $C^{k,\alpha}\subset C^0$.  The total integrand is estimated in a subspace of $L^1$, the optimal space for bounding an integral.  Curvature is given by the second derivative of the metric and the present bound is $systole^{-1}$ in the direction of gradients of small geodesic-lengths and uniform boundedness parallel to the compactification divisor.  These statements are generally optimal by the curvature expansion of \cite[Theorem 15]{Wlcurv}.

Bounds for Weil-Petersson curvature were first derived in the work of Trapani \cite{Trap}.   Detailed analysis of curvature and comparisons to other metrics were presented in the work of Liu-Sun-Yau, including considering the fourth derivative of the metric by considering the negative Ricci form as a metric and computing its Ricci curvature \cite{LSY1}.  Huang was the first to consider the dependence of curvature on the tangent direction \cite{Zh2}.  Expansions for the Riemann curvature tensor were developed in \cite{Wlbhv}.  In the work of Burns-Masur-Wilkinson \cite[Chapter 5]{BMW}, the Bers embedding, comparisons of metrics and McMullen's quasi Fuchsian reciprocity were applied to show that the $k^{th}$ derivative of the Weil-Petersson metric tensor is bounded as $systole^{-k}$. In \cite{BMMW, BMMW1}, the authors use the present estimates for norms in establishing at most polynomial mixing for non-exceptional moduli spaces and exponential mixing for exceptional moduli spaces for the Weil-Petersson geodesic flow.  

It is my pleasure to thank Rafe Mazzeo for his assistance.   

\section{Spaces of differentials, Teichm\"{u}ller space and the Weil-Petersson metric.} 
   
We review several matters, closely following the exposition and notation of \cite{Wlbers}.  The first is the description of symmetric tensors on a surface; the discussion covers Banach spaces of sections, their invariant derivatives, and the harmonic Beltrami differentials.  The second matter is the description of the tangent and cotangent spaces of Teichm\"{u}ller space, as well as the definition of the Weil-Petersson metric.  We also set our conventions for differentiating functions defined on Banach spaces.

The first matter is symmetric tensors on a Riemann surface.  Start with a compact Riemann surface $R$ possibly with punctures, uniformized by the hyperbolic plane; a goal of the present work is to extend the original considerations of \cite{Wlbers} to the case of surfaces with punctures.  Consider the canonical bundle $\kappa$ of the underlying compact surface $\bar R$.  If $\{(U_{\beta},z_{\beta})\}$ is an atlas for $\bar R$ and $\sigma$ a section of $\kappa$ with local expression $\sigma_{\beta}$ on $U_{\beta}$, then $\sigma_{\beta}\,dz_{\beta}/dz_{\gamma}=\sigma_{\gamma}$ on the overlap $U_{\beta}\cap U_{\gamma}$.  Fix a square root $\kappa^{1/2}$ of the canonical bundle of $\bar R$.  For $p,q$ and $r$ integers, let $S(p,q)$ be the space of measurable sections of $\kappa^{p/2}\otimes\bar\kappa^{q/2}$ and $S(r)=S(r,-r)$.  Example spaces are Beltrami differentials, the bounded sections in $S(-2,2)$, holomorphic quadratic differentials with at most simple poles at punctures, sections of $S(4,0)$, and the $R$-hyperbolic metric $\Lambda$, a section of $S(2,2)$.

Most of our attention will be spent on the spaces $S(r)$ since, for $\sigma\in S(2p,2q)$ and $\Lambda$ the hyperbolic metric, $\sigma\otimes\Lambda^{-(p+q)/2}\in S(p-q)$.  A basic feature of $\sigma\in S(r)$ is that its absolute value $|\sigma|$ is a function on $R$. Metric (essentially the covariant) derivatives $K_r:S(r)\rightarrow S(r+1)$ and $L_r:S(r)\rightarrow S(r-1)$ are defined as follows.  For a local coordinate $z$ and $\Lambda = (\lambda(z)|dz|)^2$, then $K_r=\lambda^{r-1}(\partial/\partial z)\lambda ^{-r}$ and $L_r=\lambda^{-r-1}(\partial/\partial \bar z)\lambda^r$.  The operators have a number of formal properties: $K_r=\overline{L_{-r}}$ and $D_r=4L_{r+1}K_r + r(r+1)=4K_{r-1}L_r+r(r-1)$ is the $\bar\partial$-Laplacian on $S(r)$.       

If we choose an atlas $\{(U_{\beta},z_{\beta})\}$ with charts designated at each puncture, then we can define for each non negative integer $k$ and each fractional index $\alpha$, $0<\alpha<1$, a $C^{k,\alpha}$ norm $\|\ \|_{k,\alpha}$ for sections of $S(r)$ in terms of the $K_*$ and $L_*$ derivatives and the H\"{o}lder estimates relative to the charts and distance at most unity.  Denote by $S_{k,\alpha}(r)$ the Banach space of  $C^{k,\alpha}$ $S(r)$-differentials.  We also consider the $C^0$ norm $\|\ \|_0$ and corresponding Banach space for bounded sections of $S(r)$.  We further consider the Banach spaces of $L^1$, $L^2$ and $L^{\infty}$ sections relative to the hyperbolic area element. 

The hyperbolic metric $\Lambda$ in a neighborhood of a puncture is isometric to 
$((|z|\log |z|)^{-1}|dz|)^2$ in a neighborhood of the origin.  A meromorphic quadratic differential in a neighborhood of the origin with poles possible only at the origin is locally in $L^1$, $L^{\infty}$, $C^{k,\alpha}$ or the $k^{th}$ Sobolev space only if the pole has order at most one.  The harmonic Beltrami differentials $B_{Harm}(R)$ are the sections $\sigma\in S(-2)$ with $\sigma=\Lambda^{-1}\otimes \bar q$ for $q$ a holomorphic quadratic differential on $R$ with at most simple poles at punctures.  Since $B_{Harm}(R)$ is finite dimensional, the various norms are  equivalent for a given surface.  In Section \ref{compare}, we consider the dependence of norm comparison on the surface.  In Lemma \ref{norm1}, we show that the $C^{k,\alpha}$ norm is bounded by the supremum norm independent of the surface. The kernel of $K_{-2}$ on smooth $S(-2)$ sections consists of sections with the form $\sigma=\Lambda^{-1}\otimes \bar q$ for $q$ holomorphic on the complement of the punctures.  Consequently the kernel of $K_{-2}$ in $L^1$, $L^{\infty}$, $C^{k,\alpha}$ or the $k^{th}$ Sobolev space is precisely the space $B_{Harm}(R)$.    

We will also be interested in the Hilbert space structure for $S(r)$ given by the Hermitian product $\langle \mu,\nu\rangle=\int_R\mu\bar\nu \Lambda$ for $\Lambda$ the $R$-hyperbolic area element. The orthogonal projection onto the harmonic differentials $P:L^2\cap S(-2)\rightarrow 
B_{Harm}(R)$ will be basic to our considerations.  On the subspace $L^{\infty}$, the projection can be given as an integral operator \cite{Ahsome}.  For our purposes, the description of $P$ simply as the orthogonal projection onto the kernel of $K_{-2}$ is sufficient.  Our main consideration in Section \ref{projector} is the variation of the projection $P$.  

The next matter is to describe the tangent and cotangent space of the Teichm\"{u}ller space $T(R)$ at the point $R$. The complex cotangent space is isomorphic to $Q(R)$, the space of holomorphic quadratic differentials with at most simple poles at punctures.  To describe the complex tangent space, first consider $B(R)=S_{2,\alpha}(-2)$, the space of $C^{2,\alpha}$ Beltrami differentials (this Banach space is required for our considerations) and also consider the pairing $B(R)\times Q(R)\rightarrow \mathC$ given by $(\mu,\varphi)=\int_R\mu \varphi$.  Now for $N(R)=Q(R)^{\perp}$ the quotient $B(R)/N(R)$ is isomorphic to the complex tangent space of $T(R)$ at $R$ and the induced pairing $B(R)/N(R)\times Q(R)\rightarrow\mathC$ is the tangent-cotangent pairing \cite{Ahsome}. 

In fact the kernel of the harmonic projection $P$ on $B(R)$ is precisely $N(R)$, thus $P:B(R)\rightarrow B_{Harm}(R)$ induces an isomorphism $P:B(R)/N(R)\rightarrow B_{Harm}(R)$.  In particular the tangent space of $T(R)$ at $R$ is isomorphic to $B_{Harm}(R)$, the space of harmonic Beltrami differentials vanishing at punctures.

The Weil-Petersson hermitian metric for Teichm\"{u}ller space is simply
\[
\langle\mu,\nu\rangle\,=\,\int_R\mu\bar\nu \Lambda\qquad\mbox{for } \mu,\nu\in B_{Harm}(R).
\]
An important matter for understanding the manifold structure of $T(R)$ is the description of a local coordinate chart and the associated coordinate vector fields.  The basic consideration is to parameterize the local deformations of $R$.  Start with an atlas for $R$, $\{(U_{\beta},z_{\beta})\}$, $z_{\beta}:U_{\beta}\rightarrow\mathC$ and a Beltrami differential $\mu$, with supremum norm $\|\mu\|_0<1$.  Let $\mu_{\beta}$ be the local representation of $\mu$ on $z_{\beta}(U_{\beta})$ and $f_{\beta}$ a homeomorphism solution of $f_{\beta,\bar z}=\mu_{\beta}f_{\beta,z}$ on $z_{\beta}(U_{\beta})$, then $\{(U_{\beta},f_{\beta}\circ z_{\beta})\}$ is an atlas for a new surface $R^{\mu}$.  Now to parameterize a neighborhood of $R$ in $T(R)$: choose $\mu_1,\dots,\mu_n$ giving a basis for $B(R)/N(R)$, then for $t=(t_1,\dots,t_n)$  small $\mu=\mu(t)=\sum_{j=1}^nt_j\mu_j$ satisfies $\|\mu(t)\|_0<1$ and $R^{\mu}=R^{\mu(t)}$ is a deformation of $R$; $t$ is a holomorphic coordinate for a neighborhood of $R$. 

\begin{lemma}\textup{\cite{Ahsome}}\label{Lmu} Notation as above.  The coordinate vector field $\partial/\partial t_j$ is given at the point $t$ by the Beltrami differential $L^{\mu(t)}\mu_j=L^{\mu}\mu_j\in B(R^{\mu})$ where
\[
L^{\mu}\mu_j\,=\,\bigg(\frac{\mu_j}{1-|\mu|^2}\frac{f_{\beta,z}}{\overline{f_{\beta,z}}}\bigg)\circ(f_{\beta})^{-1} \quad\mbox{on }U_{\beta}.
\]
\end{lemma}
The result of the lemma is summarized in the following diagram of quasiconformal maps and Beltrami coefficients.
\[
\xymatrixrowsep{4pc}\xymatrixcolsep{4pc}
\xymatrix{ R^{\mu} \ar[r]^{f^{\epsilon L^{\mu}\mu_j}} & R^{\mu +\epsilon \mu_j} \\ 
R \ar[u]^{f^{\mu}} \ar[ur]_{f^{\mu+\epsilon \mu_j}} & }
\]
Thus for $\mu=\mu(t)$, $P^{\mu}$ the $R^{\mu}$ harmonic Beltrami projector and $\Lambda^{\mu}$ the $R^{\mu}$ hyperbolic area element, the Weil-Petersson metric tensor is locally given as 
\begin{equation}\label{wpdefn}
g_{i\bar j}(t)\,=\,\langle P\mmu L\mmu \mu_i,P\mmu L\mmu \mu_j\rangle\mmu\,=\,\int_{R\mmu}P\mmu L\mmu\mu_i\,\overline{P\mmu L\mmu\mu_j}\,\Lambda\mmu,\quad \cite{Ahsome}.
\end{equation}

\section{The variation of the hyperbolic area element.}

In the work of Ahlfors and Bers \cite{AB}, the solution of the Beltrami equation is given in terms of the special potential theory operators
\[
Pg(z)\,=\,\frac{1}{2\pi i}\int g(\zeta)\big(\frac{1}{\zeta-z}-\frac{1}{\zeta}\big)d\zeta d\bar \zeta \ \mbox{ and }\ Tg(z)\,=\,\frac{1}{2\pi i}\int\frac{g(\zeta)-g(z)}{(\zeta-z)^2}d\zeta d\bar\zeta,
\]
with $(Pg)_{\bar z}=g$ and $(Pg)_z=Tg$. The operators are bounded in terms of $L^p$ norms, $p>2$.  We consider an alternative approach to the Beltrami equation through the prescribed curvature equation.  The bounds on solutions are given in terms of $C^{k,\alpha}$ and $L^2$ norms. The bounds are obtained by combining the interior Schauder estimates, estimates for the Green's operator $(D-2)^{-1}$ and the implicit function theorem.  The resulting bounds are independent of the Riemann surface.  

We take up the approach and notation of \cite[Chapter 5]{Wlhyp} and \cite[Section 3]{Wlbers} for the variation of the hyperbolic metric.  Let $R$ be a compact Riemann surface possibly with punctures with local coordinate $z=x+iy$, hyperbolic area element $\Lambda=\lambda(z)^2dxdy$ and $\mu$ a Beltrami differential with $\|\mu\|_0<1$.  Let $f^{\mu}$ be the quasiconformal homeomorphism with Beltrami differential $\mu$, and $f^{\mu}:R\rightarrow R^{\mu}$.  The $f^{\mu}$ pullback of the $R^{\mu}$ hyperbolic element $\Lambda^{\mu}$ is given simply as
\[
(f^{\mu})^*\Lambda^{\mu}\,=\,e^{2h}\Lambda, \quad \mbox{where}\quad D_*h\,-\,C_*\,=\,e^{2h},
\]
where $D$ and $C$ are the Laplacian and curvature of the $J$-pushforward of $\Lambda$ on $R^{\mu}$ and $D_*$ is the pullback of $D$ and $C_*=C\circ f^{\mu}$, with  
\[
 D_*\,=\,4\lambda^{-2}\big((1+|\mu|^2)A\frac{\partial^2}{\partial z\partial\bar z}\,+\,2\Re\big((\overline{\partial(\mu)}A)\frac{\partial}{\partial z}\,-\,A\frac{\partial}{\partial z}(\mu\frac{\partial}{\partial z})\big)\big),
\]
and
\[
C_*\,=\,-\frac12D_*\log(\lambda^2A)\,+\,4\lambda^{-2}\Re\big(-\overline{\partial(\mu)}
(\bar\mu\mu_z A)\,+\,\bar\mu(\mu_z)^2A\,+\,\mu_{zz}\big)
\]
\[
\mbox{for}\quad A\,=\,(1-|\mu|^2)^{-1},\quad 
\partial(\mu)\,=\,\frac{\partial}{\partial z}\,-\,\bar\mu\frac{\partial}{\partial \bar z}\quad\mbox{and}\quad
\overline{\partial(\mu)}\,=\,\frac{\partial}{\partial \bar z}\,-\,\mu\frac{\partial}{\partial  z}.
\]
The operator $D_*$ is the pullback Laplacian on functions; $C_*$ is the curvature of the pullback metric and $h$ is the solution of the prescribed curvature $-1$ equation.  Harmonic Beltrami differentials vanish at cusps and consequently for $\mu$ harmonic the curvature $C_*$ limits to $-1$ at punctures.  For $\mu$ also small, the curvature $C_*$ is strictly negative and in this case the prescribed curvature equation $D_*h\,-\,C_*\,=\,e^{2h}$ has a unique complete metric solution. 

\begin{lemma}\textup{\cite[Lemma 3.1]{Wlbers}}\label{metricvar} The solution $h\in S_{2,\alpha}(0)$ of $D_*h\,-\,C_*\,=\,e^{2h}$ is a real analytic function of $\mu\in S_{2,\alpha}(-2)$ for $\mu$ small.  Each derivative of $h$ with respect to $\mu$ at the origin is uniformly bounded in terms of $C^{2,\alpha}$ norms independent of the surface.
\end{lemma}  
\begin{proof} The argument is an application of the analytic implicit function theorem \cite[Section 3.3]{Berg}. Consider first the map of Banach spaces $F:S_{2,\alpha}(-2)\times S_{2,\alpha}(-2)\times S_{2,\alpha}(0)\rightarrow S_{0,\alpha}(0)$ given by 
$F(\mu,\bar\mu,h)=D_*h-C_*-e^{2h}$.  The above formulas provide that the map is real analytic.  To obtain a complex analytic map, complexify by replacing $\bar\mu$ with an independent variable $\rho\in S_{2,\alpha}(2)$.  Now $F:S_{2,\alpha}(-2)\times S_{2,\alpha}(2)\times S_{2,\alpha}(0)\rightarrow S_{0,\alpha}(0)$ given by $F(\mu,\rho,h)$ is complex analytic.

We are interested in $F$ in a neighborhood of the origin; at the origin $D_*$ is the $R$-hyperbolic Laplacian on functions, the curvature is $C_*=-1$ and $F(0,0,0)=0$.  The differential of $F$ at $(0,0,0)$ in the variable $h$ is the linear map $F_h:S_{2,\alpha}(0)\rightarrow S_{0,\alpha}(0),\,h\rightarrow (D-2)h$.  We now use a standard argument to show that the map $F_h$ is invertible.

A particular solution $f$ of $(D-2)f=(D-2)h$ for $h\in S_{2,\alpha}(0)$ is given by integrating $(D-2)h$ against an appropriate Green's function on the universal cover $\mathbb H$ \cite[pg. 155, the first example]{Fay}.  The (non automorphic) Green's function $G$ has a logarithmic pole and is integrable on the universal cover.  In particular $f=\int_{\mathH}G(D-2)h\,\Lambda$ is a $C^2$ solution of $(D-2)f=(D-2)h$ and $|f|\le\|(D-2)h\|_0\, \int_{\mathH}|G|\Lambda$.  By the standard argument the function $f$ is invariant under the uniformization group of $R$.  It now follows that $(f-h)$ is a bounded function in the kernel of $(D-2)$.  It follows either by the maximum principle for $R$ with the barrier function $\Im z$ for the puncture $\mathH/\langle z\mapsto z+1\rangle$ or by the analysis of \cite[I, Lemma 4.5]{Wlspeclim} that $(f-h)=0$.  In particular $|h|\le C\|(D-2)h\|_0$ with a universal constant.

By the interior Schauder estimate $\|h\|_{2,\alpha}\le C (\|h\|_0\,+\,\|(D-2)h\|_{0,\alpha})$ \cite{GT} and combining with the above estimate yields $\|h\|_{2,\alpha}\le C \|(D-2)h\|_{0,\alpha}$.  The operator $F_h$ is invertible and the equation $D_*h-C_*=e^{2h}$ has a solution for $\mu$ small.

The function $h$ is defined by a fixed equation with uniformly bounded quantities.  The derivatives of $C_*$ and $D_*$ are uniformly $C^{0,\alpha}$ bounded as claimed.  We proceed by induction on the derivative order of $h$.  The derivative equation is 
$(D-2)h^{(r)}\,=\,\tau$, where $\tau$ is a combination of derivative terms of order less than $r$ in $h$ and a derivative of $C_*$.  We apply the induction hypothesis to find that $\tau$ is uniformly bounded in $C^{0,\alpha}$.  By the above Green's function argument, we obtain a $C^0$ bound for $h^{(r)}$ and by the interior Schauder estimate we obtain the desired $C^{2,\alpha}$ bound on $h^{(r)}$.
\end{proof}

\section{Variation of the harmonic Beltrami projector.}\label{projector}

Our considerations also involve the potential theory for the Hilbert and Sobolev spaces of sections.  In particular for $\mu,\nu\in S(r)$, the Hermitian pairing
\[
\int_R\mu\overline{\nu}\Lambda
\]
provides the subspace $L^2(r)\subset S(r)$ of square-integrable sections with a Hilbert space structure.  The standard integration by parts formulas
\[
\langle K_r\mu,\nu\rangle\,=\,-\langle\mu,L_{r+1}\nu\rangle\quad\mbox{and}\quad\langle L_r\mu,\eta\rangle\,=\,-\langle\mu,K_{r-1}\eta\rangle
\]
for smooth $\mu\in S(r), \nu\in S(r+1)$ and $\eta\in S(r-1)$ with $\mu\overline{\nu}$ and $\mu\overline{\eta}$ having compact support, provide for defining the weak $K_*$ and $L_*$ derivatives. For example, $\mu\in S(r)$ has weak $K_r$ derivative $\psi\in S(r+1)$ provided for all smooth sections $\phi\in S(r+1)$ with compact support
\[
\langle\mu,L_{r+1}\phi\rangle\,=\,-\langle\psi,\phi\rangle.
\]
Given the context, we write $K_r\mu=\psi$.  Higher order weak derivatives are defined by compositions of differentiations with the total  order being the count of operators. 

\begin{definition} For $k$ a positive integer, the Sobolev space $H^k(r)\subset S(r)$ consists of sections with $L^2$ weak $K_*$ and $L_*$ derivatives up to total order $k$.
\end{definition} 
The $L^2(r)$ products provide a Hilbert space structure for $H^k(r)$.  For example, $\mu\in H^2(-1)$ has norm squared
\begin{multline*}
\langle\mu,\mu\rangle+\langle K_{-1}\mu,K_{-1}\mu\rangle+\langle L_{-1}\mu,L_{-1}\mu\rangle +\\
\langle K_0K_{-1}\mu,K_0K_{-1}\mu\rangle+\langle L_0K_{-1}\mu,L_0K_{-1}\mu\rangle+
\langle L_{-2}L_{-1}\mu,L_{-2}L_{-1}\mu\rangle
\end{multline*} 
(noting the commutation rule $L_0K_{-1}=K_{-2}L_{-1}+\frac12$).  

A completely general property is that smooth sections with compact support are dense in the Sobolev spaces. Our considerations for the variation of the harmonic Beltrami projector $P$ will involve uniform bounds for the Green's operator $(D_{-1}-2)^{-1}$.    
 
\begin{lemma} For $r=-1,0,1$, the operator $(D_r-2)$ has an inverse that is an isomorphism 
from $L^2(r)$ to $H^2(r)$.  The inverse is bounded independent of the surface. 
\end{lemma}
\begin{proof}  We consider the case $r=-1$; the three cases are similar.  The first step is to show that the norm $\langle f,f\rangle+\langle D_{-1}f,D_{-1}f\rangle$ is uniformly equivalent to the standard norm on $H^2(-1)$.   The issue is to bound the derivatives of order at most two in terms of only the norm of the section and its Laplacian.  Begin with considering smooth sections with compact support and the formulas
\[
\langle L_{-1}\phi,L_{-1}\phi\rangle\,=\,-\langle K_{-2}L_{-1}\phi,\phi\rangle\,=\,-\langle (D_{-1}-2)\phi,\phi\rangle
\]
and
\[
\langle K_{-1}\phi,K_{-1}\phi\rangle\,=\,-\langle L_0K_{-1}\phi,\phi\rangle\,=\,-\langle D_{-1}\phi,\phi\rangle.
\]
It follows that the norms of $L_{-1}\phi$ and $K_{-1}\phi$ are uniformly bounded by the norms of $\phi$ and $D_{-1}\phi$.  Next by further integrations by parts and the commutation rule $K_rL_{r+1}=L_{r+2}K_{r+1}+\frac{r+1}{2}$, we have the formulas
\begin{multline*}
\langle K_0K_{-1}\phi,K_0K_{-1}\phi\rangle\,=\,-\langle L_1K_0K_{-1}\phi,K_{-1}\phi\rangle\,=\\
\,-\langle K_{-1}L_0K_{-1}\phi,\phi\rangle\,=\,\langle L_0K_{-1}\phi,L_0K_{-1}\phi\rangle
\end{multline*}
and
\begin{multline*}
\langle L_{-2}L_{-1}\phi,L_{-2}L_{-1}\phi\rangle\,=\,-\langle K_{-3}L_{-2}L_{-1}\phi,L_{-1}\phi\rangle\,=\\
\langle (L_{-1}K_{-2}-1)L_{-1}\phi,L_{-1}\phi\rangle\,=\,\langle L_{-1}\phi,L_{-1}\phi\rangle+\langle K_{-2}L_{-1}\phi,K_{-2}L_{-1}\phi \rangle.
\end{multline*}
It follows that the norms of $K_0K_{-1}\phi$ and $L_{-2}L_{-1}\phi$ are uniformly bounded by the norms of $\phi$ and $(D_{-1}-2)\phi$.   We have that the norm squared $\langle \phi,\phi\rangle+\langle D_{-1}\phi,D_{-1}\phi\rangle$ is uniformly comparable to the standard norm squared on $H^2(-1)$. 

The Green's operator for $(D_{-1}-2)$ is the unique bounded operator $G:L^2(-1)\rightarrow L^2(-1)$ satisfying $(D_{-1}-2)\circ G=Id$ on $L^2(-1)$.   Fay gives the construction of the operator in his first example of \cite[pg. 155]{Fay}.  The restriction $r=-1,0,1$ comes from the consideration of the poles of the resolvent.  The relation $(D_{-1}-2)\circ G=Id$ provides that $G$ is a bounded operator to the $L^2(-1)$ subspace with norm $\langle f,f\rangle+\langle D_{-1}f,D_{-1}f\rangle$. The norm is comparable to the Sobolev norm; $G:L^2(-1)\rightarrow H^2(-1)$ is a bounded operator.  It only remains to show that $G$ is an isomorphism.   It is equivalent to show that the bounded operator $(D_{-1}-2):H^2(-1)\rightarrow L^2(-1)$ is an isomorphism.  For $\phi$ a smooth section with compact support, since $D_{-1}-2=4L_0K_{-1}-2$ then  
\[
\langle (D_{-1}-2)\phi,(D_{-1}-2)\phi\rangle\,=\,4\langle \phi,\phi \rangle -16\Re \langle L_0K_{-1}\phi,\phi\rangle+16\langle L_0K_{-1}\phi,L_0K_{-1}\phi\rangle
\]
and since $-\langle L_0K_{-1}\phi,\phi\rangle=\langle K_{-1}\phi,K_{-1}\phi\rangle$ then
\[
4\langle \phi,\phi\rangle \le \langle (D_{-1}-2)\phi,(D_{-1}-2)\phi\rangle
\]
and consequently $(D_{-1}-2)$ is an isomorphism.  The inequality and the relation $(D_{-1}-2)\circ G=Id$ combine to provide that $(D_{-1}-2)^{-1}$ is uniformly bounded from $L^2(-1)$ to $H^2(-1)$.
\end{proof} 

An application is a uniform bound for the solution of the potential equation $K_{-2}f=g$.
\begin{lemma}\label{K2solve}
The equation $K_{-2}f=g$ has a unique solution $f\in H^1(-2)$ with $f\perp B_{Harm}(R)$ and norm  uniformly bounded by the norm of $g\in L^2(-1)$.  
\end{lemma}
\begin{proof}
We apply the Green's operator for $(D_{-1}-2)=4K_{-2}L_{-1}$ to find a particular solution $F\in H^2(-1)$ such that $K_{-2}L_{-1}F=g$.  The norm of $F$ is uniformly bounded by the norm of $g$.  The quantity $4L_{-1}F$ is the desired solution of the equation.  
\end{proof}
 
The next consideration is change of variables. We follow \cite[Chapter 4]{Wlbers}.  Let $\mu$ be a harmonic Beltrami differential and let $\Lambda^{\mu}$ be the $R^{\mu}$ hyperbolic metric.  Let $z$ be a local coordinate on $R$, $w$ a local coordinate on $R^{\mu}$ and $w=f(z)$ a local representation of the quasiconformal homeomorphism with differential $\mu$.  We write $\eta(w)\in S(r)^{\mu}$ for a $(r/2,-r/2)$ symmetric tensor on $R^{\mu}$, and $f^*_r\eta$ for its {\em pullback} $\eta(f) f^{r/2}_z/\overline{f^{r/2}_z}\in S(r)$ (we fixed a square root of the $R$ canonical bundle; by continuity for $\mu$ small, we have a square root for the $R^{\mu}$ canonical bundle and for $f_z$).  Let $\partial(\mu)$ be the local coordinate operator $\frac{\partial}{\partial z}-\bar\mu\frac{\partial}{\partial \bar z}$.  

\begin{lemma}\textup{\cite[Lemma 4.1]{Wlbers}}\label{pullk} For $\eta\in S(r)^{\mu}$ a tensor on $R^{\mu}$, the pullback of the operator $K_r^{\mu}$ on $R^{\mu}$ is given as
\begin{multline*}
f^*_{r+1}(K_r^{\mu}\eta)=(1-|\mu|^2)^{-1/2}\big((f^*\Lambda^{\mu})^{(r-1)/2)}\partial(\mu)((f^*\Lambda^{\mu})^{-r/2}f_r^*\eta) \\+\frac{r}{2}(f^*\Lambda^{\mu})^{-1/2}f_r^*\eta\,\partial(\mu)
\log((\overline{f_z})^2(1-|\mu|^2))\big).
\end{multline*}
\end{lemma}

The pullback of the operator can be expressed in terms of the operators $K_r$ and $L_r$ and an order zero operator by substituting the relation $f^*\Lambda^{\mu}=e^{2h}\Lambda$, to express the pullback metric in terms of the base metric. The pullback is a first-order differential operator with coefficients the pullback hyperbolic metric and its first derivative as well as the quantity $\log((\overline{f_z})^2(1-|\mu|^2))$ and its first derivative.  Lemma \ref{metricvar} provides uniform $C^{2,\alpha}$ bounds for the $\mu$-derivatives of the pullback hyperbolic metric.  Earle studied the $C^{k+1,\alpha}$ variation of normalized solutions of the Beltrami equation for $C^{k,\alpha}, k>0$ Beltrami differentials \cite{Earbelt}.   His Theorem 2 provides analytic dependence and uniform $C^{2,\alpha}$ bounds for the $\mu$-derivatives of $\log((\overline{f_z})^2(1-|\mu|^2))$.   In summary, the pullback $K_r^{\mu}$ depends analytically on $\mu$ and there are uniform bounds in terms of the $C^{2,\alpha}$ norm of $\mu$ for the $\mu$-derivatives of the coefficients. 

The principal step to analyzing Weil-Petersson variation is to characterize the $R^{\mu}$ harmonic representatives of the coordinate vector fields $L^{\mu}\mu_j$; see Lemma 1 and the description of the harmonic Beltrami projection.  A first-order variation direct calculation was provided in Theorem 2.9 of \cite{Wlchern}.  Now our approach is to give defining equations and apply the analytic implicit function theorem.  This is the approach of Lemma 4.3 of \cite{Wlbers}.  

We seek families 
$\{\omega_j^{\mu}\}_{j=1}^n$, for $n$ the Teichm\"{u}ller space dimension, of Beltrami differentials $H^1(-2)$ such that
\[
L^{\mu}\omega_j^{\mu}\in B_{Harm}(R^{\mu}),
\]
the $L^{\mu}$ images are $R^{\mu}$ harmonic Beltrami differentials.  Equivalently the $K^{\mu}_{-2}$ derivatives vanish
\[
K^{\mu}_{-2}L^{\mu}\omega_j^{\mu}\,=\,0.
\]
Provided with the existence of the families $\{L^{\mu}\omega_j^{\mu}\}$, we apply an induction argument to show that the $\mu$-derivatives are appropriately bounded.  We then apply the Gram-Schmidt procedure to obtain an $R^{\mu}$ orthonormal frame.  The harmonic Beltrami projection is given in terms of the orthonormal frame.  We use the $C^{2,\alpha}$ and $L^2$ potential theory bounds to appropriately bound the $\mu$-derivatives of the families $\{L^{\mu}\omega_j^{\mu}\}$.   The bounds will carry over to the harmonic Beltrami projection and the Weil-Petersson product. 

We specialize and begin with an {\em orthonormal basis} $\{\mu_j\}_{j=1}^n$ for $B_{Harm}(R)$.   The families $\{\omega^{\mu}_j\}_{j=1}^n$ of Beltrami differentials $B(R)$ will satisfy
\begin{equation}\label{products}
\langle \omega^{\mu}_j,\mu_k\rangle\,=\,\delta_{jk}
\end{equation}
and
\begin{equation}\label{harmonic}
f_{-1}^*K^{\mu}_{-2}L^{\mu}\omega_j^{\mu}\,=\,0.
\end{equation}

\begin{lemma}\label{omegavar}\textup{\cite[Lemma 4.3]{Wlbers}} There exist Beltrami differentials $\omega^{\mu}_j\in S_{2,\alpha}(-2),$ $j=1,\dots,n$, depending analytically on $\mu\in B_{Harm}(R)\cap\stwoalpha$, such that for $\mu$ small, $\{\omega_j^{\mu}\}^n_{j=1}$ satisfy equations \textup{(\ref{products})} and \textup{(\ref{harmonic})}.  In particular the differentials $L\mmu\omega_j\mmu$ are harmonic on $R\mmu$.   The quantities $L^{\mu}\omega_j^{\mu}$ varying in $H^1(-2)$ have uniformly bounded derivatives at the origin as functions of $\mu\in B_{Harm}(R)\cap\stwoalpha$.
\end{lemma}
\begin{proof}  We first consider that the map
\[
F:(\omega_j,\mu)\,\longrightarrow\,\big(\langle \omega_j,\mu_k\rangle ,f_{-1}^*K^{\mu}_{-2}L^{\mu}\omega_j\big)
\]
from $H^1(-2)^n\times S_{2,\alpha}(-2)$ to $\mathbb C^{n^2}\times L^2(-1)^n$ is real analytic for $\mu$ small.  The inner products are constant in $\mu$.  As discussed above, the pullback operator $f^*_{-1}K_{-2}^{\mu}$ has coefficients real analytic in $C^{1,\alpha}$ for $\mu$ in 
$\stwoalpha$.   The evaluation $f^*_{-1}K_{-2}^{\mu}L^{\mu}\omega_j$ is the sum of $\omega_j$ and its first derivatives with coefficients varying real analytically in $C^{1,\alpha}$.  The evaluation is real analytic in $L^2(-1)$ since $\omega_j\in H^1(-2)$.  The map $F=F(\omega_j,\mu,\bar\mu)$ is complex linear in $\omega_j$ and real analytic in $\mu$.  If we replace $\bar\mu$ by the independent variable $\rho$, the resulting map will be complex analytic. 

To apply the analytic implicit function theorem, it remains to check the differential of $F$ in $(\omega_j)$ at the origin $\mu=\rho=0$. The differential is the map 
$(\omega)\rightarrow (\langle\omega,\mu_k\rangle, K_{-2}\omega)$ from $H^1(-2)$ to $\mathbb C^n\times L^2(-1)$. The inner products provide an invertible map from $B_{Harm}(R)$ to $\mathbb C^n$ and, by Lemma \ref{K2solve}, $K_{-2}$ is an invertible map from $H^1(-2)\cap B_{Harm}(R)^{\perp}$ to $L^2(-1)$.   The differential of $F$ is invertible.  The frame $\{\omega^{\mu}_j\}_{j=1}^n$ exists by the analytic implicit function theorem applied to the equation $F(\omega_j,\mu,\bar\mu)=0$, \cite[Section 3.3]{Berg}.

We now bound the $\mu$ derivatives of $\omega^{\mu}_j$ by induction. We write $\|\ \|_{C^{2,\alpha}},$ $\|\ \|_{L^2}$ and $\|\ \|_{H^k}$ for the $C^{2,\alpha}$, the square-integrable and the Sobolev $k$-norms for sections $S(r)$. For a quantity $Q$, depending analytically on $\mu$, we write $Q^{(k)}$ for any of its $k^{th}$ derivatives; our considerations apply equally for real and complex derivatives.  The induction statement is that for each non negative integer $k$, there is a positive constant $C_k$, independent of the surface, such that $\|(\omega^{\mu}_j)^{(k)}\|_{H^1} \le C_k \|\mu\|_{C^{2,\alpha}}^k$.  For $k=0$, equations (\ref{products}) and (\ref{harmonic}) provide that $(\omega^{\mu}_j)^{(0)}=\mu_j$ with $\mu_j$ an orthonormal basis element of $B_{Harm}(R)$.   By an integration by parts $\|\mu_j\|^2_{H^1}=2\|\mu_j\|^2_{L^2}$ and the first estimate is established.  We have from the equations in general that $(\omega^{\mu}_j)^{(k+1)}\perp B_{Harm}(R)$ and for the $\mu$ derivative evaluated at zero,
\[
(f^*_{-1}K^{\mu}_{-2}L^{\mu}\omega^{\mu}_j)^{(k+1)}=K_{-2}(\omega^{\mu}_j)^{(k+1)}\,+\,\sum_{p=1}^{k+1}{k+1 \choose p}(f_{-1}^*K_{-2}^{\mu}L^{\mu})^{(p)}(\omega^{\mu}_j)^{(k+1-p)}=\,0.
\]
As already observed, $(f_{-1}^*K_{-2}^{\mu}L^{\mu})^{(p)}$ is a first-order differential operator with coefficients uniformly bounded in $C^{1,\alpha}$ by $C'\|\mu\|_{C^{2,\alpha}}^p$.  By the induction, the variational term $(\omega^{\mu}_j)^{(k+1-p)}$ has norm uniformly bounded in $H^1$ as  $\|(\omega^{\mu}_j)^{(k+1-p)}\|_{H^1}\le C'' \|\mu\|_{C^{2,\alpha}}^{k+1-p}$.   It follows that the above sum is uniformly bounded in $L^2$ by $\|\mu\|_{C^{2,\alpha}}^{k+1}$.  By Lemma \ref{K2solve}, we conclude that $(\omega^{\mu}_j)^{(k+1)}$ is uniformly bounded in $H^1$ by $\|\mu\|^{k+1}_{C^{2,\alpha}}$, as desired.   
\end{proof}

We are ready to bound the harmonic Beltrami projections of the coordinate vector fields (see Lemma \ref{Lmu}). 
\begin{theorem}\label{main1} Let $\{\mu_j\}_{j=1}^n$ be an orthonormal basis for $B_{Harm}(R)$.  The harmonic Beltrami projection of the coordinate vector fields $L\mmu\omega_j\mmu$ varying in $L^2(R\mmu)$ have uniformly bounded derivatives at the origin as functions of $\mu\in B_{Harm}(R)\cap S_{2,\alpha}(-2)$.   In particular the pullbacks $f_{-2}^*L\mmu\omega_j\mmu$ vary with uniformly bounded derivatives in $L^2(R)$.  The Weil-Petersson coordinate pairings $g_{jk}(\mu)=\langle P\mmu L\mmu \mu_j,P\mmu L\mmu\mu_k\rangle\mmu$ are real analytic with uniformly bounded derivatives at the origin as functions of $\mu\in B_{Harm}(R)\cap S_{2,\alpha}(-2)$.
\end{theorem} 
\begin{proof}  The key is the behavior of the projection $P\mmu$.  Begin with the Beltrami differentials specified in Lemma \ref{omegavar} and apply the Gram-Schmidt process for the pairing $\langle\ ,\ \rangle\mmu$, to obtain an orthonormal basis.  The orthonormal basis is a linear transform of the original basis and the operator $L\mmu$ is linear; the new basis has the form $\{L\mmu\tilde\omega_j\mmu\}^n_{j=1}$. Each new basis vector $\tilde\omega_j\mmu$ is a rational linear expression in the inner products $\langle L\mmu\omega_p\mmu,L\mmu\omega_q\mmu\rangle\mmu$ and a linear combination of the $L\mmu\omega_k\mmu$.  The pairings have initial values $\delta_{pq}$.  Since the Beltrami differentials $\{L\mmu\tilde\omega_j\mmu\}_{j=1}^n$ are orthonormal and $R\mmu$ harmonic, the projection $P\mmu$ is simply given as $P\mmu(\nu)=\sum_{j=1}^n\langle\nu,L\mmu\tilde\omega_j\mmu\rangle\mmu L\mmu\tilde\omega_j\mmu$ and the Weil-Petersson coordinate pairing is given as $g_{jk}(\mu)=\langle P\mmu L\mmu\mu_j,P\mmu L\mmu\mu_k\rangle\mmu$.  The desired uniform bounds for the derivatives follow provided we show that $L\mmu\omega\mmu_j$ varying in $L^2(R\mmu)$, $f_{-2}^*L\mmu\omega_j\mmu$ varying in $L^2(R)$, 
$\langle L\mmu\omega_j\mmu,L\mmu\omega_k\mmu\rangle\mmu$ and $\langle L\mmu\mu_j,L\mmu\omega_k\mmu\rangle\mmu$ each have uniformly bounded derivatives at the origin as functions of $\mu\in B_{Harm}(R)\cap S_{2,\alpha}(-2)$.   

We consider the bounds. The pullback maps $f^*_{-2}$ provide trivializations for the Hilbert bundle $L^2\cap S\mmu(-2)$.  For the first two quantities it suffices to show that $f_{-2}^*L\mmu\omega_j\mmu=(1-|\mu|^2)^{-1}\omega_j\mmu$ varies appropriately in $L^2(R)\cap S(-2)$ for $\mu$ small.   This statement follows immediately from Lemma \ref{omegavar}.  For the integrals use the quasiconformal map $f\mmu$ to pullback integration to $R$.   We have the integrals
\[
\langle L\mmu\omega_j\mmu,L\mmu\omega_k\mmu\rangle\mmu\,=\,\int_R(1-|\mu|^2)^{-2}\omega_j\mmu\overline{\omega_k\mmu}\,(f\mmu)^*\Lambda\mmu
\]
and
\[
\langle L\mmu\mu_j,L\mmu\omega_k\mmu\rangle\mmu\,=\,\int_R(1-|\mu|^2)^{-2}\mu_j\overline{\omega_k\mmu}\,(f\mmu)^*\Lambda\mmu.
\]
By Lemmas \ref{metricvar} and \ref{omegavar}, the integrals are real analytic and have uniformly bounded derivatives at the origin as functions of $\mu\in B_{Harm}(R)\cap S_{2,\alpha}(-2)$.  
\end{proof}

We do not expect the quantities $L\mmu\omega_j\mmu$ varying in $\stwoalpha$ to be uniformly bounded, since the recursion for the terms $(L\mmu\omega_j\mmu)^{(k)}$ begins with $\mu_j$ which by Lemma \ref{norm2} below is not bounded in $\stwoalpha$ independent of the surface.

\section{Comparing norms for harmonic Beltrami differentials.}\label{compare}

We relate the $C^0$, $C^{2,\alpha}$, $L^2$ and $H^k$ norms for harmonic Beltrami differentials.  We consider the $C^0$ and $C^{2,\alpha}$ norms below and then review our direct comparison of the $C^0$ and $L^2$ norms from \cite[Corollary 11]{Wlcurv}.  The $L^2$ and Sobolev $k$-norms are related by integration by parts.

\begin{lemma}\label{norm1} Given a non negative integer $k$ and $0<\alpha<1$, there is a positive constant $C_{k,\alpha}$ such that the norms of a harmonic Beltrami differential satisfy $\|\mu\|_0\le\|\mu\|_{k,\alpha}\le C_{k,\alpha}\|\mu\|_0$.
\end{lemma}
  
\begin{proof} The first inequality is immediate.  We introduce a cover of the surface and reduce the matter to uniform bounds for standard coordinate discs and punctured discs.   A cover of the surface is given by uniformizing punctured discs $\{0<|z|\le 3/4\}$ with hyperbolic metric $\rho^{-1}|dz|^2$ for $\rho=(|z|\log |z|)^2$ and by uniformizing discs $\{|z|\le 3/4\}$ with hyperbolic metric $\rho^{-1}|dz|^2$ for $\rho=(1-|z|)^2/4$.   A union of radius $1/2$ discs and punctured discs covers the surface. 
 
We first compare the norms for a disc.  The $C^0$ norm of a harmonic Beltrami differential $\overline{\varphi}\rho$ on a disc bounds the holomorphic function $\varphi$ on the radius $3/4$ circle. By the Cauchy Integral Formula the derivative of $\varphi$ of each order is bounded on the radius $1/2$ disc by the bound for $\varphi$ on the circle of radius $3/4$.  Now the differentiation operators are $K_r=\rho^{\frac{r+1}{2}}\frac{\partial}{\partial \overline{z}}\rho^{\frac{-r}{2}}$ and $K_r=\overline{L_{-r}}$.  An order $k$ derivative of $\overline{\varphi}\rho$ is a sum of products of the derivatives of $\overline{\varphi}$ and $\rho$.  Each derivative of $\overline{\varphi}\rho$ is uniformly bounded on the radius $1/2$ disc in terms of the bound for the radius $3/4$ circle.  The derivative bounds are also used to bound the H\"{o}lder $\alpha$-quotient. The $C^{k,\alpha}$ norm for the radius $1/2$ disc is uniformly bounded.    

We start with observations for the punctured disc.  The quadratic differential $\varphi$ can have a simple pole at the origin and the $n^{th}$ derivative of $\rho$ has magnitude $O(|z|^{-n}\rho)$, since the maximal term is the $n^{th}$ derivative of the logarithm.  Since the differentiation operator $L_r$ has homogeneity exponent $1/2$ in the quantity $\rho$, we observe that $L_{-n}\cdots L_{-2}(\overline{\varphi}\rho)$ is a sum of monomials each with a net exponent $\frac{n+1}{2}$ in $\rho$ and unit exponent in $\overline{\varphi}$; each monomial includes $n-1$ derivatives.  Each derivative increases the magnitude by multiplying by $|z|^{-1}$; so the $p^{th}$ derivative of $\overline{\varphi}$ has magnitude bounded by $|\varphi z^{-p}|$ and the $p^{th}$ derivative of $\rho$ has magnitude bounded by $|\rho z^{-p}|$.   Combining all the considerations it follows that $L_{-n}\cdots L_{-2}(\overline{\varphi}\rho)$ is uniformly pointwise bounded on the radius $1/2$ punctured disc by $|z||\log|z||^{n+1}\max_{|z|=3/4}|\overline{\varphi}\rho|$, a quantity vanishing at the puncture.  We also consider $K_r$ derivatives.  The quantity $\overline{\varphi}\rho$ is annihilated by $K_{-2}$ and the commutation rules provide that the $K_{n-1}$ derivative of $L_{-n}\cdots L_{-2}(\overline{\varphi}\rho)$ is a set multiple of $L_{-n+1}\cdots L_{-2}(\overline{\varphi}\rho)$. In summary, the derivatives of $\overline{\varphi}\rho$ on a punctured disc of radius $1/2$ are bounded by the maximum of the quantity of the radius $3/4$ circle.  Again the derivative bounds can be used to bound the H\"{o}lder $\alpha$-quotient.  The $C^{k,\alpha}$ norm on a punctured disc is uniformly bounded by the maximum on the circle of radius $3/4$.  
\end{proof}

A point of the Teichm\"{u}ller space $T(R)$ represents a Riemann surface $R$ and the equivalence class of an isomorphism of the fundamental group of a reference surface.  For a free homotopy class $[\alpha]$ of a non trivial, non peripheral closed curve on the reference surface, the geodesic-length $\ell_{\alpha}(R)$ is the length of the unique closed geodesic on $R$ in the corresponding free homotopy class.  Associated to the geodesic-length are the differential $d\ell_{\alpha}$ and the root-length gradient $\lambda_{\alpha}=\grad \ell_{\alpha}^{1/2}$ \cite{Wlbhv,Wlcurv}. For a pants decomposition, a maximal collection $\{\alpha_j\}_{j=1}^{3g-3+n}$ of simple, disjoint geodesics, the root-length gradients $\{\lambda_{\alpha_j}=\grad \ell_{\alpha_j}^{1/2}\}_{j=1}^{3g-3+n}$ form a basis for $B_{Harm}(R)$ \cite[Theorem 3.7]{WlFN}.  For $\ell_{\alpha},\ell_{\beta}$ small, the geodesics $\alpha,\beta$ are simple, disjoint and the gradients $(2\pi)^{1/2}\lambda_{\alpha},(2\pi)^{1/2}\lambda_{\beta}$ are approximately unit length and approximately orthogonal; in particular 
\begin{equation}\label{prodexp}
\langle \lambda_{\alpha},\lambda_{\beta}\rangle\,=\,\frac{\delta_{\alpha\beta}}{2\pi}\ +\ O(\ell_{\alpha}\ell_{\beta}^2)
\end{equation}
for $\ell_{\alpha}\le\ell_{\beta}\le c$, the Kronecker delta and the $O$ constant depending on $c$ \cite[Lemma 3.12]{Wlbhv}. For $c_0$ sufficiently small, the set $\sigma=\{\alpha\mid \ell_{\alpha}\le c_0\}$ of short geodesics consists of simple, disjoint geodesics with $\{(2\pi)^{-1/2}\lambda_{\alpha}\}_{\alpha\in\sigma}$ approximately an orthonormal collection.  In particular $\mu\in B_{Harm}(R)$ has a unique representation
\[
\mu\,=\,\sum_{\alpha\in\sigma}a^{\alpha}\lambda_{\alpha}\ +\ \mu_0 \quad\mbox{with}\quad \mu_0\perp\lambda_{\alpha},\alpha\in\sigma
\]
and $2\pi a^{\alpha}\approx \langle\mu,\lambda_{\alpha}\rangle$.  To describe the $C^0$ to $L^2$ norm comparison we introduce an invariant.
\begin{definition}\label{Ratio} For short geodesics $\sigma$ and a harmonic Beltrami differential $\mu$, define
\[
Comp(\mu)\,=\,\frac{\max_{\alpha\in\sigma}|\langle\mu,\lambda_{\alpha}\rangle|\ell_{\alpha}^{-1/2}+\|\mu_0\|_{L^2}}{\|\mu\|_{L^2}}.
\]
\end{definition}
The maximal value of the invariant is $systole^{-1/2}$, for systole the smallest geodesic-length.   
The following norm comparison is established in \cite[Corollary 11]{Wlcurv}.
\begin{lemma}\label{norm2} Setup as above.  Let $\sigma$ be the geodesics of length at most $c_0$. For $c_0$ sufficiently small, there are positive constants $c',c''$  such that for a harmonic Beltrami differential
\[
c'\, Comp(\mu)\le\, \frac{\|\mu\|_0}{\|\mu\|_{L^2}}\, \le c''\, Comp(\mu). 
\]
\end{lemma}
\begin{proof} We sketch the earlier considerations \cite{Wlcurv}.  The substantial masses of the Beltrami differentials $\lambda_{\alpha}$ and $\mu_0$ are almost disjoint.  By Corollary 9 the root-length gradient has magnitude $O(\ell_{\alpha}^{3/2})$ on the complement of the $\alpha$-collar and has maximal magnitude $2/(\pi\ell_{\alpha}^{1/2})$ on $\alpha$.  The orthogonality $\mu_0\perp\lambda_{\alpha}$ and expansion Proposition 7 for a collar provide that $|\mu_0|$ is uniformly bounded in terms of $\|\mu_0\|_{L^2}$ and is small on the $\sigma$ collars; trivially $\|\mu_0\|_{L^2}$ is bounded by $\|\mu_0\|_0$.  It follows for $c_0$ sufficiently small, that the maxima of $|\langle\mu,\lambda_{\alpha}\rangle\lambda_{\alpha}|$ and $|\mu_0|$ can be considered separately.  The desired conclusion follows from the description of the individual Beltrami differentials.   
\end{proof}

There is a simple application of the lemma to bounding the derivative of the Riemann period matrix with respect to the Weil-Petersson metric.  Recall a compact Riemann surface $R$ of genus $g$ has a canonical homology basis $\{\alpha_j,\beta_j\}_{j=1}^g$ with intersection numbers satisfying $\alpha_j\cdot\beta_k=\delta_{jk}$.  The space of holomorphic one-forms has the Hodge inner product with norm
\[
\|\omega\|_H^2\,=\,\frac{i}{2}\int_R\omega\wedge\overline\omega,\quad \omega\in\Omega(R).
\]
There is a corresponding canonical basis $\{\omega_k\}_{k=1}^g$ of $\Omega(R)$ specified by the condition
\[
\int_{a_j}\omega_k\,=\,\delta_{jk}
\]
and the Riemann period matrix is defined by
\[
\Pi_{jk}\,=\,\int_{b_j}\omega_k.
\]
For $\omega\in\Omega(R)$, with $a$-periods the column vector $(a_j)$ and $b$-periods the column vector $(b_j)$, the periods are related as
\[
(b_j)\,=\,\big(\Pi_{jk}\big)(a_k).
\]
Rauch established the variational formula for a Beltrami differential
\[
d\Pi_{jk}[\mu]\,=\,\int_R\omega_j\omega_k\mu, \quad \cite{Rauch}.
\]
We immediately have the inequality
\[
\big|d\Pi_{jk}[\mu]\big|\,\le\,c\,\|\omega_j\|_H\|\omega_k\|_H Comp(\mu)
\]
for a $L^2$ unit-norm harmonic Beltrami differential, the bound for Weil-Petersson displacement.

\section{Applications.}

We now put the derivative bounds of Theorem \ref{main1} in context with the Bers embedding and Weil-Petersson covariant derivatives.  We begin by reviewing the direct relation between harmonic Beltrami differentials, the Bers embedding and the Weil-Petersson metric.

Start with the uniformization $\mathbb H/\Gamma$ of the Riemann surface $R$.  Let $B(\Gamma)$ be the $L^{\infty}$ Banach space of $\Gamma$-invariant Beltrami differentials on $\mathbb H$ and $M(\Gamma)\subset B(\Gamma)$ its open unit ball.  Given $\nu\in M(\Gamma)$, extend the differential by zero on the lower half plane $\mathbb L$.  There exists a homeomorphism solution $W\nnu$ of the equation $W_{\bar z}=\nu W_z$ on the complex plane $\mathbb C$ \cite{AB}.  The map $W\nnu$ is holomorphic on the lower half plane; the map is unique provided it fixes $0$ and $1$.  A most important feature is that $\Gamma\nnu=W\nnu\circ\Gamma\circ(W\nnu)^{-1}$ is a Kleinian group and the map $W\nnu$ is $\Gamma-\Gamma\nnu$-equivariant.  In particular $W\nnu(\mathbb H)/\Gamma\nnu$ is a Riemann surface 
and $W\nnu$ gives a homotopy class of a map $R\rightarrow R\nnu$.  The map $W\nnu$ as an element of an appropriate Banach space depends complex analytically on $\nu\in M(\Gamma)$.  An equivalence relation on $M(\Gamma)$ is defined as: $\mu\sim\nu$ provided $W\nnu\vert_{\mathbb R}=W\mmu\vert_{\mathbb R}$. The quotient $M(\Gamma)\slash\sim$ is the Teichm\"{u}ller space of $R$.  

The Schwarzian derivative of a locally injective conformal map $f(z)$ is the quantity $\{f,z\}=(f'''/f')-(3/2)(f''/f')^2$ .   The group equivariance of the map $W\nnu$ provides that $\{W\nnu,z\}\vert_{\mathbb L}$ is a $\Gamma$ invariant holomorphic quadratic differential.  The quotient $\mathbb L/\Gamma$ is the conjugate Riemann surface $\bar{R}$. The map $M(\Gamma)\rightarrow Q(\bar R)$ factors through the projection to equivalence classes and gives rise to an injection of Teichm\"{u}ller space, the Bers embedding,
\[
T(R)\rightarrow Q(\bar R)\quad\mbox{given by the mapping}\quad \nu/\sim\ \rightarrow\,\{W\nnu,z\},
\]
\cite{Bersembed}. The image is a bounded domain.  

The range space of the Bers embedding can be provided with the supremum metric: for  $\varphi\in Q(\bar R),\, \|\varphi\|_0=\max |\varphi|\Lambda^{-1}$ for $\Lambda$ the hyperbolic area element. The classical necessary and sufficient conditions for a conformal injection provide that the Bers embedding contains the $\|\ \|_0$-ball of radius $1/2$ and is contained in the $\|\ \|_0$-ball of radius $3/2$, \cite{Lehtobk}.  Ahlfors-Weill described a section of the Bers embedding \cite{AhWe}.  For $\psi\in Q(\bar R)$ with $\|\psi\|_0<1/2$, then $\{W^{-2\overline{\psi(\bar z)}\Lambda^{-1}},z\}=\psi$ - the harmonic Beltrami differential $-2\overline{\psi(\bar z)}\Lambda^{-1}\in M(\Gamma)$ maps to the quadratic differential $\psi$ by the Bers embedding. Accordingly, we can also consider the Bers embedding as mapping to the space of harmonic Beltrami differentials with norm $\|\ \|_0$.   

A manifold local coordinate $z\in\mathbb C^n$ is $k$-normal for a K\"{a}hler metric $g$, provided the power series expansion in $z$ of $g$ has vanishing pure-$z$ and pure-$\bar z$ terms for total homogeneity terms at most $k$.   A coordinate is $\omega$-normal provided it is $k$-normal for all positive integers $k$.  The hyperbolic metric for a disc is $\omega$-normal.  Given a K\"{a}hler metric $g$, a tangent frame $\mathcal F$ at a point $p$, then provided it exists, the $\omega$-normal coordinate $z$ with $z(p)=0$ and initial frame $\mathcal F$ is unique \cite{Gilkbk}. 

The main theorem of \cite[Theorem 4.5]{Wlbers} and the above Theorem \ref{main1} are combined to the following.  The norm $\|\ \|_0$ provides a displacement scale for derivatives. 
\begin{theorem}   The Bers embedding provides $\omega$-normal coordinates for the Weil-Petersson metric. For an initial orthonormal basis of harmonic Beltrami differentials, the derivatives of the Weil-Petersson metric tensor at the origin of the embedding are uniformly bounded depending only on the differentiation order. 
\end{theorem} 

We follow Bochner's approach \cite{Boch} and present formulas in terms of the complexification of the tangent space.  Bochner observed for a K\"{a}hler metric 
\[
ds^2\,=\,2g_{\alpha\overline{\beta}}dt_{\alpha}d\overline{t_{\beta}}
\]
that the standard formulas simplify (as always repeated indices are summed).  The Christoffel symbols for the covariant derivative involve only unmixed components
\[
\Gamma^{\alpha}_{\beta\gamma}\,=\,g^{\alpha\overline{\sigma}}\frac{\partial g_{\overline{\sigma}\beta}}{\partial t_{\gamma}}
\]
and the components of the curvature tensor are
\[
R_{\alpha\overline{\beta}\gamma\overline{\delta}}\,=\,\frac{\partial^2g_{\alpha\overline{\beta}}}{\partial t_{\gamma}\partial\overline{t_{\delta}}}\,-\,g^{\rho\overline{\sigma}}
\frac{\partial g_{\rho\overline{\beta}}}{\partial\overline{t_{\delta}}}\frac{\partial g_{\overline{\sigma}\alpha}}{\partial t_{\gamma}}.
\]
The K\"{a}hler condition provides the symmetries
\[
R_{\alpha\overline{\beta}\gamma\overline{\delta}}\,=\,R_{\gamma\overline{\beta}\alpha\overline{\delta}}\,=\,R_{\alpha\overline{\delta}\gamma\overline{\beta}}.
\]
We continue to assume that the initial tangent frame is orthonormal.  The formulas for a general frame are obtained by applying linear transformations.   The final bounds are presented in terms of the Definition \ref{Ratio} norm comparison.

\begin{theorem}\label{main2} For each non negative integer $m$, and each non negative multi index $\mathbf k=(k_1,\dots,k_n)$, there are positive constants $C_m$ and $C_{\mathbf k}$ as follows.  Let $\{\mu_{\alpha}\}_{\alpha=1}^n$ be an orthonormal basis for $B_{Harm}(R)$.  For $t=(t_1,\dots,t_n)$ small, let $\mu=\sum_{\alpha=1}^nt_{\alpha}\mu_{\alpha}$ and $g_{\alpha\overline{\beta}}(\mu)\,=\,\langle P\mmu L\mmu\mu_{\alpha},P\mmu L\mmu \mu_{\beta}\rangle\mmu$ be the Weil-Petersson pairing for the local coordinate $t$.  The initial derivatives of the metric are bounded as 
\[
\Big|\Big(\frac{\partial}{\partial t_1}\Big)^{k_1}\cdots\Big(\frac{\partial}{\partial t_n}\Big)^{k_n} g_{\alpha\overline{\beta}}(0)\Big|\,\le\,C_{\mathbf k}\,Comp(\mu_1)^{k_1}\cdots Comp(\mu_n)^{k_n}\ \ \mbox{for all}\ \alpha,\beta.
\]
For elements $\nu_1,\dots,\nu_m\in B_{Harm}(R)$, the covariant derivatives of curvature are bounded as
\[
\Big|\mathbf D_{\nu_1}\cdots\mathbf D_{\nu_m}R_{\alpha\overline{\beta}\gamma\overline{\delta}}\Big|\,\le\,C_m\, Comp(\mu_{\gamma})\,Comp(\mu_{\delta})\,Comp(\nu_1)\cdots Comp(\nu_m).
\]
\end{theorem}
\begin{proof}Coordinate derivatives are homogeneous in the involved tangent vectors.  The coordinate derivative bounds follow from Theorem \ref{main1} by scaling the tangent vectors to unit Weil-Petersson length. The scaling is bounded by Lemma \ref{norm2}.  For the second set of bounds it is enough to consider evaluation for coordinate vector fields since the curvature and its covariant derivatives are tensors.  For a $q$-covariant tensor 
$\mathbf T$, vector fields $V_1,\dots,V_q$ and vector $v$, the covariant derivative is 
\begin{multline*}
\mathbf D_v \mathbf T(V_1,\dots,V_q)\,=\\ v(\mathbf T(V_1,\dots,V_q))\, -\, \mathbf T(\mathbf D_vV_1,\dots,V_q)\,-\,\cdots \,-\, \mathbf T(V_1,\dots,V_{q-1},\mathbf D_vV_q).
\end{multline*}
Each covariant derivative of curvature is a fixed polynomial in the coordinate derivatives of the metric tensor.  The conclusion now follows from the coordinate derivative estimate.
\end{proof}

The above provides for the general boundedness of Weil-Petersson curvature and its covariant derivatives parallel to the compactification divisor of the moduli space of Riemann surfaces.  Neighborhoods of the compactification divisor are described by collections of disjoint simple geodesics having small geodesic-lengths.  Consider such a collection of geodesics $\sigma$  and a neighborhood where the geodesic-lengths $\ell_{\alpha},\, \alpha\in\sigma$ are small.  In the augmented Teichm\"{u}ller space $\overline T$, the locus $\{\ell_{\alpha}=0\mid\alpha\in\sigma\}$ defines a boundary product of lower dimensional Teichm\"{u}ller spaces $T(\sigma)$ \cite[Section 4.1]{Wlbhv}.  A relative length basis for a point of $T(\sigma)$ is a collection $\tau$ of free homotopy classes, disjoint from the elements of $\sigma$, such that the geodesic-length functions $\ell_{\beta},\,\beta\in\tau$ provide local coordinates for a neighborhood of the point on $T(\sigma)$ \cite[Definition 4.1]{Wlbhv}.  In a neighborhood of the point in $\overline T$, the gradients $\{\grad \ell_{\alpha}^{1/2},i\grad \ell_{\alpha}^{1/2}, \grad \ell_{\beta}^{1/2}\}_{\alpha\in\sigma,\beta\in\tau}$ provide a $\mathbb R$-frame for the tangent bundle.  It follows from the inner product expansion (\ref{prodexp}), that for $\ell_{\beta}\le c,\,\beta\in\tau$, there is a constant $c'$ such that $Comp(\grad\ell_{\beta})\le c',\,\beta\in\tau$.  From the theorem, the Weil-Petersson metric and its derivatives are bounded for all tangents in the span of the relative length basis gradients.    

    
\providecommand\WlName[1]{#1}\providecommand\WpName[1]{#1}\providecommand\Wl{Wlf}\providecommand\Wp{Wlp}\def\cprime{$'$}

\end{document}